\documentclass{elsarticle}
\usepackage[utf8]{inputenc}
\usepackage[margin=1.25in]{geometry}
\usepackage{bbm}

\usepackage{amsmath,amsthm,amsfonts,amssymb, mathtools}
\usepackage{mathtools}
\usepackage{hyperref}
\usepackage{enumerate}
\usepackage{graphicx}
\usepackage{relsize}
\usepackage{algorithm}
\usepackage{algpseudocode}
\usepackage{xcolor} 
\usepackage{verbatim}
\usepackage[labelfont=bf]{caption}
\hypersetup{
    colorlinks=true,
    linkcolor=blue,
    citecolor=blue,
    urlcolor=red
    pdftitle={}
}

\newtheorem{theorem}{Theorem}[section]
\newtheorem{lemma}[theorem]{Lemma}

\newtheorem{proposition}[theorem]{Proposition}
\newtheorem{corollary}[theorem]{Corollary}
\newtheorem{conjecture}[theorem]{Conjecture}

\theoremstyle{definition}
\newtheorem{definition}[theorem]{Definition}
\newtheorem{example}[theorem]{Example}

\newtheorem{observation}[theorem]{Observation}

\newcommand{\R}{\mathbb{R}}

\newcounter{figures}
\newcommand{\figlabel}[1]{%
   \refstepcounter{figures}%
   \label{#1}}

\graphicspath{{./figures/}}

\begin{document}

\begin{frontmatter}
\title{Classifying Tractable Instances of the Generalized Cable-Trench Problem}

\author[add1]{Mya Davis}
\author[add2]{Carl Hammarsten\corref{cor}}
\ead{carl.hammarsten@desales.edu}
\author[add3]{Siddarth Menon}
\author[add4]{Maria Pasaylo}
\author[add5]{Dane Sheridan}

\cortext[cor]{Corresponding author}
\address[add1]{Wake Forest University}
\address[add2]{DeSales University}
\address[add3]{University of California, Berkeley}
\address[add4]{University of Florida}
\address[add5]{University of California, Los Angeles}

\date{\today}

\begin{abstract}
Given a graph $G$ rooted at a vertex $r$ and weight functions, $\gamma, \tau: E(G)\to \R$, the generalized cable-trench problem (CTP) is to find a single spanning tree that simultaneously minimizes the sum of the total edge cost with respect to $\tau$ and the single-source shortest paths cost with respect to $\gamma$ \cite{Vasko1, Vasko2}. 
Although this problem is provably $NP$-complete in the general case, we examine certain tractable instances involving various graph constructions of trees and cycles, along with quantities associated to edges and vertices that arise out of these constructions. 
We show that given a graph in which all cycles are edge disjoint, there exists a fast method to determine a cable-trench solution. 
Further, we examine properties of graphs which contribute to the general intractability of the CTP and present some open questions in this direction. 
\end{abstract}


\nonumnote{This project was supported by NSF Grant DMS-1852378 and DMS-2150299, REU Site: Moravian University Computational Methods in Discrete Mathematics.}

\end{frontmatter}


\section{Introduction}


We begin with a connected graph $G=(V, E)$, where $V$ is the set of vertices and $E$ is the set of edges. Each edge is represented by pairs of the vertices they adjoin. We call $G$ a weighted graph if we have a weight function $\gamma: E\to \R$. 
The \textit{minimum spanning tree} problem and the \textit{single-source shortest paths} problem are problems in the study of combinatorial algorithms with efficient and well-studied solutions.

The \textit{minimum spanning tree} problem is the problem of finding a connected acyclic graph $T = (V', E')$ such that $V'=V$ and $E'\subseteq E$ (referred to as a spanning tree) such that the sum of the weights of each edge in $E'$, with respect to some weight function $\gamma$, is minimized over all possible spanning trees. The total weight of such a $T$ is called the \textit{cost} of that minimum spanning tree, denoted $W_{MST}(T,\gamma)$.
This problem is efficiently solved by Prim's algorithm, in which the spanning tree is constructed by iteratively adding the vertex $v$ such that the cost of adding $v$ is minimized. This is accomplished via a priority queue in which high priorities are assigned to low-cost vertices.

The \textit{single-source shortest paths} problem specifies a distinguished vertex, called the root $r$, and asks to find a spanning tree $T$ for which the total weight of a root path from $r$ to every vertex $w\in V$, with respect to weight function $\gamma$, is minimized for all $w$ (by convention, the length of the path from $r$ to $r$ is $0$). The total weights of all root paths in such a $T$ is called the \textit{cost} of that single-source shortest paths tree, denoted $W_{SPT}(T,\gamma)$.
The single-source shortest paths problem is efficiently solved by Dijkstra's algorithm. Similarly, vertices are placed in a priority queue and vertices are added to the shortest paths tree as the distance of the corresponding root path is updated. 

A note on uniqueness. In both the minimum spanning tree and single-source shortest path tree problems the specific tree attaining minimum cost may not be unique. However, the minimum attainable value of the cost function \emph{is} unique. In this paper, we focus on attaining this unique minimum cost. Even when an explicit construction of a tree attaining such a value is given, we do not mean to imply that this solution is unique, only that the cost associated is indeed minimal.

Given a graph $G$ with a single weight function $\gamma$, the \textit{single constraint cable-trench problem} (SCTP) (introduced in \cite{Vasko1}) asks for a spanning tree $T$ that minimizes some linear combination of the cost of a minimum spanning tree $W_{MST}(T,\gamma)$ and the cost of the single-source shortest paths tree $W_{SPT}(T,\gamma)$ with respect to a root $r$.
The cable-trench cost $cost_{SCTP}(T,\gamma)$ is the linear combination with real coefficients $\alpha$ and $\beta$ defined by the following equation:

$$cost_{SCTP}(T,\gamma) = \alpha W_{MST}(T,\gamma) + \beta W_{SPT}(T,\gamma)$$

The name ``cable-trench'' comes from interpreting this problem as finding the minimum cost to construct a computer network across a large campus. Specifically, we let the root vertex $r$ stand for a central server and all other vertices denote remote work-stations. Now, each work-station needs a cable directly to the central server (i.e. a root path). Minimizing the cost of these cables yields the single-source shortest tree path problem. However, we also need to bury all of the cables, so we need to dig trenches between buildings on campus. Minimizing the cost of these trenches yields the shortest path tree problem. Assuming the cost of purchasing a cable, or digging a trench, is proportional to it's length, we minimize a linear combination of the costs of each individual problem and get the single constraint cable-trench problem.

This problem was shown to be $NP$-complete by noting that given any vertices $s,t$, finding a minimum spanning tree for which the $s-t$ path length is minimal is $NP$-hard \cite{Vasko1}.
Consequently, the cable-trench problem which minimizes the length of all $r-w$ paths is certainly $NP$-complete.
In this case however, there are reasonably efficient and effective algorithms that approximate optimal solutions depending on the ratio $\frac{\alpha}{\beta}$.

In this article, we focus on the \textit{generalized cable-trench problem} which is a variation of the cable-trench problem in which there are two independent weights assigned to each edge. 
That is, there are functions, $\gamma: E\to \R$, $\tau: E\to \R$, and the problem is to find a spanning tree which minimizes the sum of the weights of the single-source shortest paths tree with respect to $\gamma$ and the minimum spanning tree with respect to $\tau$. 
Following the traditional interpretation of this problem with respect to physical network construction, $\gamma$ is called the cable cost function and $\tau$ is called the trench cost function. Because edge weights can always be normalized with respect to $\alpha$ and $\beta$, we may omit the coefficients entirely. 
Subsequently, we aim to minimize the generalized cable-trench cost $cost(T)$ function:
$$cost(T) = W_{MST}(T, \tau) + W_{SPT}(T, \gamma)$$

As a reminder, in this paper we focus on finding a tree $T$ that attains the unique minimum value of $cost(T)$ without any claim that $T$ itself is unique. Informally, we use both \textbf{cable-trench solution} and \textbf{cable-trench tree} to refer to any minimum-cost spanning tree $T$ with respect to edge weight functions $\gamma$ and $\tau$.

 It immediately follows that the generalized cable-trench problem is also $NP$-complete because of the complexity of the single constraint variation. 
 
 Vasko et al. \cite{Vasko2} formulated the mixed integer linear program for the cable-trench problem as follows:

Minimize:
$$\sum\sum\gamma_{ij}x_{ij}+\sum\sum\tau_{ij}y_{ij}$$

subject to

\begin{align*}
    &\sum x_{1j}=n-1 && i=1\\
    &\sum x_{ij}-\sum x_{ki} =-1 && i=2, 3, \dots\\
    &\sum y_{ij}=n-1 && \text{all edges }i<j\\
    &(n-1)y_{ij}-x_{ij}-x_{ji}\geq 0 && \text{all edges }i<j\\
    &x_{ij} \geq 0 && \forall i,j\\
    &y_{ij} \in \{0, 1\} && \forall i,j.
\end{align*} 

Here, $x_{ij}$ denotes the number of cables from vertex $i$ to vertex $j$. Since each cable is a root path originating at $r=v_1$, we interpret the cable as ``oriented'' away from $v_1$ and say it runs \emph{from} vertex $i$ to vertex $j$ if $v_i$ is closer to $v_1$ under this orientation of the cable.
Next, $y_{ij}$ is 1 if and only if a trench is dug between vertices $i,j$. 
The first constraint thus ensures that exactly $n-1$ cables leave the root vertex $r=v_1$.
The second set of constraints ensures that exactly one root path terminates at every vertex $v_i$, other than the root. That is, every vertex has a cable laid \textit{to it} from the root. 
The third set of constraints ensures that $n-1$ trenches are dug. 
The fourth set of constraints ensure that cables are laid where trenches are dug, and the last two constraints guarantee positivity and integrality of the linear program variables.

\subsection{Heuristic algorithms and inapproximability}

Significant research has been done into heuristic algorithms that can approximate optimal spanning tree solutions to the cable-trench problem. 
In the general case, given the similarity of Dijkstra's solution to the single-source shortest paths problem and Prim's solution to the minimum spanning tree problem, a natural idea is to modify the algorithm such that vertices are added to a spanning tree with priority relative to the cable-trench cost that is added to the tree. 
Vasko et al.  \cite{Vasko2} analyzed such a modified algorithm in the case of large cable-trench networks and found that it reasonably approximated efficient solutions.

In \cite{Khullar}, Khuller,  Raghavachari, and Young describe an algorithm CREW-PRAM which computes a spanning tree with a continuous tradeoff between minimum spanning tree cost and single-source shortest paths tree cost. 
For a given $\lambda > 0$, the algorithm approximates a minimum spanning tree up to a factor of $1+\sqrt{2}\lambda$ and a single-source shortest paths tree up to a factor of $1+\frac{\sqrt{2}}{\lambda}$. 
Further work done by Benedito, Pedrosa, and Rosado \cite{Benedito} show that there exists an $\epsilon > 0$ for which approximating an optimal solution up to a factor of $1+\epsilon$ is NP-hard.

\section{Definitions}

\begin{definition}
    Recall that a \textbf{spanning tree} of a graph $G = (V(G), E(G))$ is a connected, acyclic graph $T = (V(T), E(T))$ such that $V(T) = V(G)$ and $E(T)\subseteq E(G)$. 
\end{definition}

\begin{definition}
    Given any (sub)graph $G$, we define the \textbf{size} $|G| = |V| - 1$. That is, $|G|$ denotes the number of edges in any spanning tree of $G$.
\end{definition}

Note, if $G$ is a tree itself, then $|G|$ = $|E|$ as well.

\begin{definition}
    Given a path $P$, we define the \textbf{trench length} $\tau(P)$ as the weighted length of $P$ with respect to the edge-weight function $\tau: E(G)\to \R$.
    $$\tau(P) = \sum_{e\in P}\tau(e)$$
\end{definition}

\begin{definition}
    Given a path $P$, we define the \textbf{cabling length} $L(P)$ as the weighted length of $P$ with respect to the edge-weight function $\gamma: E(G)\to \R$.
    $$L(P) = \sum_{e\in P}\gamma(e)$$
\end{definition}

\begin{definition}
    Given a path $P$ with an enumeration of the edges from one end to the other, $P: (e_1, \dots, e_n)$, we define the \textbf{cabling cost}, \emph{with respect to this orientation}, $C(P)$ as the total cabling length of all distinct subpaths through $P$ starting from $e_1$.
    $$C(P) = \sum_{i = 1}^{n} \sum_{j = 1}^{i}\gamma(e_j)$$
\end{definition}

\begin{lemma}\label{lem:ConcatenatedPreCabling} 
    Given oriented paths $A,B$, such that $A$ and $B$ are edge-disjoint with the final vertex on $A$ overlapping the first vertex on $B$, we can compute the cabling cost of the path concatenation $AB$ as follows:
        $$C(AB) = C(A) + L(A)|B| + C(B)$$
\end{lemma}

\begin{proof}
Assume path $AB$ comes with oriented edge list $(e_1, \dots , e_n)$ where path $A = (e_1, \dots , e_k)$ and path $B = (e_{k+1}, \dots , e_n)$. The costs of cabling path $A$ and path $B$ are 
$$C(A) = \sum_{i = 1}^{k} \sum_{j = 1}^{i}\gamma(e_j)$$
$$C(B) = \sum_{i = k+1}^{n} \sum_{j = k+1}^{i}\gamma(e_j)$$
and lengths of paths $A$ and $B$ are 
$$L(A) = \sum_{i = 1}^k\gamma(e_i)$$
$$L(B) = \sum_{i = k+1}^n\gamma(e_i)$$

By observation, 
$$L(AB) = L(A) + L(B)$$

Then,
\begin{align*}
    C(AB) &= \sum_{i = 1}^{n} \sum_{j = 1}^{i}\gamma(e_j)\\
    &= \sum_{i = 1}^{k} \sum_{j = 1}^{i}\gamma(e_j) + (n-k)\sum_{i = 1}^{k}\gamma(e_j) + \sum_{i = k+ 1}^{n} \sum_{j = k+1}^{i}\gamma(e_j)\\
    &= C(A) + |B|L(A) + C(B)\\
    &= C(A) + L(A)|B| + C(B)
\end{align*}
\end{proof}

Essentially, we account for the cabling cost of $A$ in the first term $C(A)$. 
Every cable that starts at $A$ and ends in $B$ must pay the entire cabling cost of $A$, after which it is cabled normally, and is accounted for by the cabling cost of $B$ in the last term $C(B)$. 
We thus think of the $L(A)|B|$ term as \textbf{pre-cabling} path $B$ through $A$, which yields the recursive structure shown in Figure \ref{fig:precabling}.

\begin{figure}[hbtp]
    \figlabel{fig:precabling}
    \begin{center}
    \includegraphics[width=300pt]{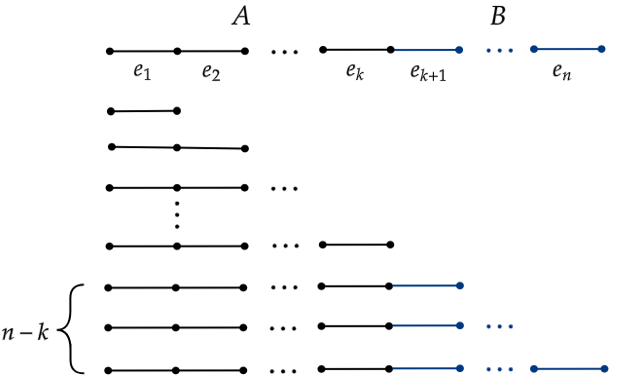}
	\caption{Cabling $B$ through $A$ requires laying $k$ cables of length $\sum_{i=1}^n e_i$ to reach the endpoint of $A$, as well as the additional cabling cost of path $B$. The cabling cost of concatenating paths $A$ and $B$ requires pre-cabling of path $B$ through path $A$.}
    \end{center}

\end{figure}

Lastly, we define two graph operations central to our study of the cable-trench problem.

\begin{definition}
    Given graphs $G, H$, and vertices $v_G$ on $G$ and $v_H$ on $H$, we define the \textbf{wedge} $G\wedge H$ (omitting the vertex information for brevity) as the graph formed by combining $G, H$ such that vertices $v_G, v_H$ are identified as the same vertex.
\end{definition}

\begin{definition}
	Given $G \wedge H$ and a spanning tree $T$ in this wedge graph, we define the \textbf{restriction} of $T$ to $G$ as the graph $T|_G = (V, E)$ with $V = V(G)$ and $E = E(T) \cap E(G)$.
\end{definition}


\section{Results on Wedging} {\label{section3}}
Given two graphs, $G$ and $H$, with known cable-trench solutions, we can easily compute a cable-trench solution for the wedge $G \wedge H$, if we perform the wedge operation at their respective root vertices and use this common vertex as the new root of $G \wedge H$.


\begin{proposition}{\label{wedge_trees_at_roots}}
Given graphs $G, H$ rooted at $r_G$, $r_H$ and with known cable-trench solutions $T_G$, $T_H$. A cable-trench solution for $G \wedge H$, formed by identifying $r_G, r_H$ and using this new common vertex as the root of the result, is the tree $T_G \wedge T_H$ formed by identifying $r_G, r_H$.
\end{proposition}

\begin{proof}
Since $G$ and $H$ are wedged at their common root vertex and that vertex is the new root, all rooted paths in $G \wedge H$ must be contained completely within either $G$ or $H$ and cannot cross between the two components. Thus, the cost of a spanning tree $T$, in $G \wedge H$, can be decomposed as the sum of the costs of $T|_G$ (i.e. $T$ restricted to $G$) and $T|_H$.

$$cost(T) = cost(T|_G)+cost(T|_H)$$

Furthermore, the costs of each restriction $T|_G$ and $T|_H$ are clearly independent. So, in order to minimize the sum, we only need minimize each term independently.
By assumption, these are minimized by $T_G$ and $T_H$ respectively. Thus, the weight of the minimum-cost spanning tree, in $G \wedge H$, is $cost(T_G) + cost(T_H)$. Finally, $T_G \wedge T_H$ necessarily attains this minimum and thus it is a cable-trench solution for $G \wedge H$ as desired.
\end{proof}

With Lemma~\ref{wedge_trees_at_roots} we have shown that, if the roots of $G$ and $H$ are identified in the wedge operation, and this common vertex is used as the resulting wedge graph's root, the local cable-trench solutions $T_G$ and $T_H$, in $G$ and $H$ respectively, immediately determine a global cable-trench solution in $G \wedge H$. 

This type of ``local-to-global'' information exchange will be the foundation for an inductive heuristic to compute more and more complex cable-trench solutions by decomposing our target graph into the wedge of (many) simpler parts.

To expand the class of graphs for which this local-to-global information exchange is possible, we consider the case of wedging two graphs $G$ and $H$, as above, at the root $r_H$ in $H$ but at an arbitrary vertex $v \neq r_G$ in $G$. In this case, we also let $r_G$ be the root of the wedge graph $G \wedge H$.

Following similar logic to the previous proof we have a one-sided version of Proposition~\ref{wedge_trees_at_roots}.

\begin{lemma}\label{lem:WedgeTreeRemains}
Let $G$ be any graph and $H$ be a graph with known cable-trench solution $T_H$. Consider the wedge graph $G \wedge H$, where this wedge is formed by identifying arbitrary non-root vertex $v$ in $G$ to the root vertex $r_H$ in $H$. If $T$ is a cable-trench solution for $G \wedge H$, we may assume without loss of generality that the restriction $T|_H$ is exactly $T_H$.
\end{lemma}

\begin{proof}
Since $G$ and $H$ are wedged at the root vertex of $H$, any root path $P$ in $G \wedge H$ that ends at a vertex originally from $H$ may be decomposed into two paths $P|_G$ and $P|_H$ with $P|_G$ connecting $r_G$ to $r_H$ entirely within $G$ and $P|_H$ starting from $r_H$ and terminating somewhere in $H$ while remaining entirely within $H$. So, following Lemma~\ref{lem:ConcatenatedPreCabling}, the cabling cost of each root path $P$ terminating in $H$ can be decomposed into the cabling cost of $P|_G$, the pre-cabling cost of $P|_H$ through $P|_G$, and the cabling cost of $P|_H$.

Now, given a specific choice of spanning tree $T$, in $G \wedge H$, we are given a preferred path from $r_G$ to $r_H$ and so we have that the restriction $P|_G$ is the same for all potential root paths $P$ contained in $T$, regardless of which specific vertex in $H$ is its terminus. We denote this specific root path from $r_G$ to $r_H$ by $Q$.

The combined pre-cabling cost of all root paths terminating in $H$ is then independent of the structure of $T|_H$. This fact follows because the pre-cabling cost, through $Q$, is equal for all possible root paths terminating in $H$, hence the total pre-cabling cost for all such paths depends only on how many paths we need, not which specific paths we choose. Since all spanning trees of $H$ have the same number of edges, all potential spanning trees of $G \wedge H$ have the same number of necessary root paths terminating in $H$. Thus, the combined pre-cabling cost for all of $T|_H$ is always $L(Q) | T|_H | = L(Q) |H|$.

The cable-trench cost $cost(T|_G)$ is also clearly independent from the structure of $T|_H$ since the wedge graph root is the same as $r_G$. 

Thus, the cable-trench cost of a spanning tree $T$ can be decomposed as $ cost(T|_G) + L(Q) | H | + cost(T|_H)$ and the first two terms in the sum are independent of the structure of $T|_H$ beyond the size of $H$, which is constant. Consequently, the third term $cost(T|_H)$ may be minimized independently from the choice and/or structure of $T|_G$.

By assumption, $T_H$ attains the minimum cable-trench cost in $H$ and thus it follows that setting $T|_H = T_H$ is one possible way to minimize that contribution, independent of what $T|_G$ looks like. Since we focus only on attaining the minimum cable-trench cost, without worrying about the uniqueness (or lack thereof) of the tree $T$, this gives us a guaranteed workable option as desired.
\end{proof}

With Lemma~\ref{lem:WedgeTreeRemains} we have shown that, when wedging $H$ at its root vertex and regardless of which vertex we identify it with from $G$, we may always use a local cable-trench solution $T_H$, as our presumptive restriction to $H$, in any systematic search for a global cable-trench solution in $G \wedge H$. 

Unfortunately, the local-to-global information exchange coming from $G$, where we perform the wedging operation at a non-root vertex, is necessarily more complicated. However, we may still expand upon our results in certain cases. For the remainder of this section we explore one such case and assume $G$ is a cycle.

\begin{figure}[hbtp]
    \figlabel{fig:CycleTerms}
    \begin{center}
    \includegraphics[width=200pt]{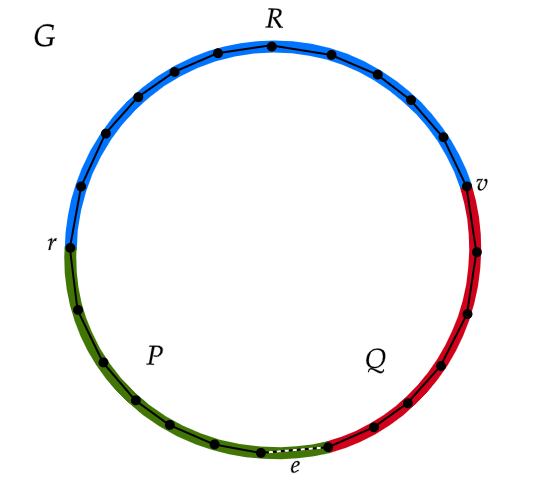}
	\caption{The cycle $G$ and distinguished edge $e$ with labeled components $P, Q,$ and $R$.}
    \end{center}
\end{figure}

Following the diagram for cycle $G$ presented in Figure \ref{fig:CycleTerms}, we define various important structures within $G$, given a specific edge of interest $e$:
\begin{itemize}
    \item $T$ will refer to the spanning tree of $G$ rooted at $r$ that does not contain edge $e$. 
    \item $R$ denotes the root path from $r$ to $v$ that does not include $e$. 
    \item $PQ$ denotes the root path from $r$ to $v$ that does include $e$. The notation is chosen because we will always consider $PQ$ as the concatenation of two paths: $P$ from $r$ to $e$, \emph{including} $e$ and $Q$ from $e$ to $v$, \emph{excluding} $e$.
    \item Note that, by construction, the three paths $P$, $Q$, and $R$ are edge disjoint and that $E(G) = P \cup Q \cup R$.
    \item For cabling length computations, $P$ and $R$ are always assumed to originate at the root. However, the cabling of $Q$ is sometimes done in each of the two possible orientations. So we use $Q^+$ to indicate the clockwise orientation and $Q^-$ to indicate the counter-clockwise orientation when necessary. Here the use of ``clockwise'' and ``counter-clockwise'' is only for expository simplicity. We only need to distinguish the two orientations as opposites.
\end{itemize}

For notation purposes, if an edge has a subscript, the corresponding components in the decomposition from Figure \ref{fig:CycleTerms} associated to this edge all share the same subscript (i.e. $e_i$ specifies spanning tree $T_i$, and splits $G$ into paths $P_i$, $Q_i$, and $R_i$. 

Given two edges, $e_1$ and $e_2$, with one on each of the two distinct paths from $r$ to $v$, we will need to compare the two trees $T_1$ and $T_2$. We denote the path, connecting from $e_1$ to $e_2$ through $v$, by $\tilde{Q} = Q_1 \cup Q_2$. Note, the orientation for $Q_1$ and $Q_2$ must be the same for $\tilde{Q}$ to be an oriented path. So, if we assume $e_1$ is on the counter-clockwise-oriented path from $r$ to $v$ and $e_2$ is on the clockwise-oriented path, then the clockwise orientation, from $e_2$ to $e_1$, on $\tilde{Q}$ arises as $\tilde{Q}^+ = Q_2^+ \cup Q_1^+$. Similarly, the counter-clockwise orientation, from $e_1$ to $e_2$, arises as $\tilde{Q}^- = Q_1^- \cup Q_2^-$. See Figure \ref{fig:multiple_wedges} for an example of such labeling.

Now, we give the following results for wedging arbitrary $H$, at its root vertex, onto cycle $G$ at a non-root vertex.

\begin{lemma}\label{lem:BetterAlpha}
    Given a cycle $G$, suppose a cable-trench solution for $G$ is the tree $T_1$ that deletes edge $e_1$ (see Figure \ref{fig:CycleTerms}). Let $H$ be an arbitrary graph with known cable-trench solution $T_H$ and consider the wedge $G \wedge H$ performed at the root $r_H$ of $H$ and some non-root vertex $v$ of $G$. For any edge $e_2 \neq e_1$ on $P_1 \cup Q_1$, consider the associated tree $T_2$. We must have $cost(T_{2} \wedge T_H) \geq cost(T_{1} \wedge T_H)$.
\end{lemma}
\begin{proof}
    Given that $T_1$ is a cable-trench solution for $G$, we know $cost(T_2) \geq cost(T_1)$. Then, observe that in each of the path decompositions of $G$ arising from $e_1$ and $e_2$ we have $R_1 = R_2$ since $e_1$ and $e_2$ are on the same $r_G$ to $v$ path. We refer to this common path in $G$ generically as $R$ and note that the total cable-trench cost of wedging $T_H$ at $v$ is $L(R)|H| + cost(T_H)$ regardless of whether we choose to delete $e_1$ or $e_2$. 
    Therefore, $cost(T_{2} \wedge H) = cost(T_2) + L(R)|H| + cost(T_H) \geq cost(T_1) + L(R)|H| + cost(T_H) = cost(T_1 \wedge H) $ as desired.
\end{proof}

Before using these lemmas to present an algorithm for computing the cable-trench solution for wedging a graph $H$ with known cable-trench solution onto a cycle $G$, we define the subroutine \textbf{CycleCTP} that takes a cycle graph $G$ as an input and returns an edge $e_1$ for which $T_1$ is a cable-trench solution for $G$ via a brute force minimization over all $n$ spanning trees of $G$. 
Consider the following algorithm that takes the following as input: a cycle $G$, a vertex $v\in G$, and a cable-trench solution $T_H$ for graph $H$ with root $r_H$ to be wedged onto $v$. Note that computing the cost of a given tree can be done in $O(n)$ time.

\newcommand{\cost}{cost}

\begin{algorithm} 
\caption{CycleWedgeCTP}
\label{alg:cap}
\begin{algorithmic}
\State $e_1 \gets \textbf{CycleCTP}(G)$
\State $\min \gets \cost(T_1 \wedge T_H)$ \Comment{Graphs wedged at $v$}
\For{edge $e_2 \in R_1$}
\If{$\cost(T_2 \wedge T_H) <$ min} \Comment{Graphs wedged at $v$}
\State $\text{min} \gets \cost(T_2 \wedge T_H)$
\EndIf
\EndFor
\State \Return $\text{min}$
\end{algorithmic}
\end{algorithm}

Notice by Lemma \ref{lem:BetterAlpha}, we must only loop over edges on path $R_1$ (that is, \emph{not} on path $P_1 \cup Q_1$), which on average enables us to only loop over about half the edges. 
Based on the given cost function and algorithm as defined in Algorithm \ref{alg:cap}, we present the following proposition as a proof of correctness.


\begin{proposition}\label{prop:inequality}
     Assume $G$ is a cycle rooted at $r$ with cable-trench solution $T_1$, the tree constructed by excluding edge $e_1$. 
     Assume $H$ is an arbitrary graph with known cable-trench solution $T_H$. 
     If we wedge $H$ at its root onto $G$ at a non-root vertex $v$, $T_1 \wedge T_H$ is a cable-trench solution for $G \wedge H$ as long as for all edges $e_2 \in E(G)$ on path $R_1$, we have:
     
    $$0 \geq \tau(e_2) - \tau(e_1) + L(P_2) - L(P_1) + (L(P_2)-L(P_1))|\tilde{Q}| + C(\tilde{Q}^+) - C(\tilde{Q}^-) + (L(R_1) - L(R_2))|H|$$
\end{proposition}

\begin{proof}
Following the convention of Figure \ref{fig:CycleTerms}, the total cable-trench cost of the spanning tree $T_1$ in $G$ is:
$$ \tau(E(G)) - \tau(e_1) + C(P_1^-) - L(P_1) + C(R_1^+) + L(R_1)|Q_1| + C(Q_1^+)$$

By Lemma~\ref{lem:WedgeTreeRemains}, we know a cable-trench solution in $G \wedge H$ may be assumed to restrict to $T_H$ in $H$. So, we consider the spanning tree $T_1 \wedge T_H$ in $G \wedge H$ and note the resulting total cable-trench cost may be computed as:
$$ \tau(E(G)) - \tau(e_1) + C(P_1^-) - L(P_1) + C(R_1^+) + L(R_1)|Q_1| + C(Q_1^+) + L(R_1)|H| + cost(T_H)$$

Here, the $L(R_1)|H|$ term arises from pre-cabling all of the cables that are internal to $T_H$ backward from the root of $H$ to the root of $G$.

By Lemma~\ref{lem:BetterAlpha}, we know only an edge $e_2$ which is on $R_1$ could possibly result in a lower cost cable-trench tree $T_2 \wedge T_H$. 
Thus for such an edge, we compute the total cable-trench cost of $T_2 \wedge T_H$ as:
$$ \tau(E(G)) - \tau(e_2) + C(P_2^+) - L(P_2) + C(R_2^-) + L(R_2)|Q_2| + C(Q_2^-) + L(R_2)|H| + cost(T_H)$$

If the cable-trench cost of excluding edge $e_1$ were to be cheaper than when excluding edge $e_2$, the difference of the above total cable-trench costs must be non-positive. Hence, after some direct cancellations and rearranging of terms, we have the following inequality:
\begin{align*}
    0 \geq\  &\tau(e_2) - \tau(e_1) \\
     +\ &L(P_2) - L(P_1)\\
     +\ &[C(R_1^+) - C(P_2^+)] - [C(R_2^-) - C(P_1^-)]\\
     +\ &L(R_1)|Q_1| - L(R_2)|Q_2|\\
     +\ &C(Q_1^+) - C(Q_2^-)\\
     +\ &L(R_1)|H| - L(R_2)|H|
\end{align*}

Now, observe that $R_1^+$ is the concatenation of $P_2^+$ and $Q_2^+$. 
That is, $R_1^+ = P_2^+Q_2^+$. 
Similarly, $R_2^- = P_1^-Q_1^-$. 
Hence, by Lemma \ref{lem:ConcatenatedPreCabling} we have:
$$C(R_1^+) - C(P_2^+)=L(P_2)|Q_2| + C(Q_2^+)$$
$$C(R_2^-) - C(P_1^-)=L(P_1)|Q_1| + C(Q_1^-)$$

And so, after plugging these in and more rearranging of terms, the inequality becomes:
\begin{align*}
    0 \geq\  &\tau(e_2) - \tau(e_1) + \\
    +\ &L(P_2) - L(P_1) \\
    +\ &[L(P_2)|Q_2| + L(R_1)|Q_1|] - [L(P_1)|Q_1| + L(R_2)|Q_2|] \\
    +\ &C(Q_2^+) - C(Q_1^-)\\
    +\ &C(Q_1^+) - C(Q_2^-) \\
    +\ &L(R_1)|H| - L(R_2)|H|
\end{align*}

Now, by the additivity of $L$, we have that $L(R_1) = L(P_2) + L(Q_2)$. 
So, recalling that we use $\tilde{Q} = Q_1 \cup Q_2$ (as sets of edges), it follows that: 
\begin{align*}
L(P_1)|Q_1| + L(R_2)|Q_2|  &= L(P_1)|Q_1| + L(P_1)|Q_2| + L(Q_1)|Q_2|  \\
&= L(P_1)|\tilde{Q}| + L(Q_1)|Q_2|
\end{align*}

Similarly, we also get:
$$L(P_2)|Q_2| + L(R_1)|Q_1| = L(P_2)|\tilde{Q}| + L(Q_2)|Q_1|$$

And hence the inequality becomes:
\begin{align*}
    0 \geq\ &\tau(e_2) - \tau(e_1)\\
    +\ &L(P_2) - L(P_1)\\
    +\ &L(P_2)|\tilde{Q}| - L(P_1)|\tilde{Q}| \\
    +\ &[C(Q_2^+) + L(Q_2)|Q_1| + C(Q_1^+)]\\
    -\ &[C(Q_1^-) + L(Q_1)|Q_2| + C(Q_2^-)]\\
    +\ &L(R_1)|H| - L(R_2)|H|
\end{align*}

Finally, noting that our orientation convention gives $\tilde{Q}^+ = Q_2^+Q_1^+$ and $\tilde{Q}^- = Q_1^-Q_2^-$, we can use Lemma \ref{lem:ConcatenatedPreCabling} again to get:
$$C(\tilde{Q}^+) = C(Q_2^+) + L(Q_2)|Q_1| + C(Q_1^+)$$
$$C(\tilde{Q}^-) = C(Q_1^-) + L(Q_1)|Q_2| + C(Q_2^-)$$

So our inequality becomes:
\begin{align*}
    0 \geq\ &\tau(e_2) - \tau(e_1)\\
    +\ &L(P_2) - L(P_1) \\
    +\ &L(P_2)|\tilde{Q}| - L(P_1)|\tilde{Q}|\\
    +\ &C(\tilde{Q}^+) - C(\tilde{Q}^-) \\
    +\ &L(R_1)|H| - L(R_2)|H|
\end{align*}
\end{proof}
\textbf{Remark:} The inequality in Proposition~\ref{prop:inequality} can be decomposed into pairs of terms that represent specific cost differences between extending $T_1$ and $T_2$ to a potential cable-trench solution for $G \wedge H$:

\begin{itemize}
    \item $\tau(e_2) - \tau(e_1)$ represents the difference in trench length. $T_1$ does not have to `dig the trench' through $e_1$, and $T_2$ does not have to dig the trench through $e_2$.
    
    \item In both $T_1$ and $T_2$, paths $P_1$ and $P_2$ have to be cabled, and they are both cabled from the same direction (originating from root $r$). 
    However, in $T_1$, the last cable along edge $e_1$ does not need to be laid, as the edge is excluded from the tree. Similarly, $T_2$ does not include the last edge on its path.
    Therefore, while the cost of cabling $P_1$ and $P_2$ cancel, $L(P_2) - L(P_1)$ remains as the costs arising from including edges $e_1$ and $e_2$ on paths $P_1$ and $P_2$ respectively.
    
    \item The term $(L(P_2) - L(P_1))|\tilde{Q}|$ represents the difference in the cost of pre-cabling $\tilde{Q}$. In $T_1$, we lay cables to $\tilde{Q}$ with respect to one orientation and thus pre-cable it through $P_2$. While in $T_2$, we lay cables to $\tilde{Q}$ with respect to the opposite orientation and thus pre-cable it through $P_1$. 
    
    \item $C(\tilde{Q}^+) - C(\tilde{Q}^-)$ accounts for the difference in cabling region $\tilde{Q}$ with respect to the two different possible orientations arising from $T_1$ and $T_2$. 
    
    \item Finally, $L(R_1)|H| - L(R_2)|H|$ represents the difference in the pre-cabling costs for graph $H$ through either $R_1$ or $R_2$.
    Note that the cost of the continuation of these cables internal to $H$ cancels out as they are the same in either case of the restriction in $G$ being $T_1$ or $T_2$. 
\end{itemize}

With Proposition~\ref{prop:inequality} we have shown that, when wedging a cycle $G$ at a non-root vertex, we may often use a local cable-trench solution $T_G$ when determining a global cable-trench solution in $G \wedge H$. Moreover, we have a precise check for when this extension of $T_G$ would fail to provide the desired minimum cost in the wedge graph. More in-depth analysis about when, and how, this particular local-to-global transfer of information fails will be explored in detail in Section~\ref{section4}.

For now, observe that our particular choice of wedge vertex $v$ only really impacts the last element of the inequality in Proposition~\ref{prop:inequality}. The remainder of the inequality is only dependent on our choices of $e_1$, $e_2$, and the different decompositions of $G$ into paths each gives. Specifically, the region $\tilde{Q}$, between $e_1$ and $e_2$ which contains the wedge vertex $v$ is evidently quite important.

As such, if we are to generalize the arguments of Proposition~\ref{prop:inequality} to the case of wedging multiple graphs onto a common cycle at multiple distinct vertices, we would again expect the majority of relevant cable-trench cost changes to arise along the same $\tilde{Q}$ path between $e_1$ and $e_2$. 

\begin{figure}[hbtp]
    \figlabel{fig:multiple_wedges}
    \begin{center}
    \includegraphics[width=200pt]{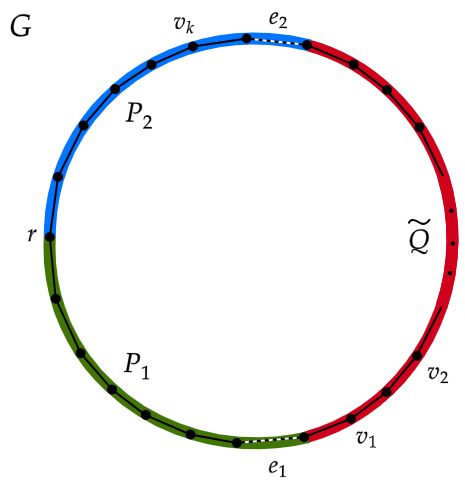}
	\caption{Wedging multiple graphs onto non-root vertices on the cycle $G$.}
    \end{center}

\end{figure}

In general, let $\mathcal{H} = \{H_1, \dots, H_k\}$ be a set of graphs for which each graph $H_i$ has known cable-trench solution. 
Let $\mathcal{V} = \{v_1, \dots, v_k\}$ be a set of vertices, where $v_i$ is the vertex on $G$ which is identified with the root of $H_i$ in the graph $G \wedge H_i$.
We define the set $H(\tilde{Q}) \subseteq \mathcal{V}$ as the set of vertices on path $\tilde{Q}$ that are contained in $\mathcal{V}$. 

Adapting notation conventions specified earlier, for an edge $e_i \in G$, denote by $R_i^{(v)}$ the path from root $r$ to $v$ that does not include $e_i$.  


\begin{theorem}\label{thm:Wedging}
     Assume $G$ is a cycle rooted at $r$ with cable-trench solution $T_1$, the tree constructed by excluding edge $e_1$. 
     Assume $H_{v_1}, H_{v_2}, \dots, H_{v_n}$ are arbitrary graphs with known cable-trench solutions $S_1, S_2, \dots, S_n$. 
     If we wedge each $H_i$ at its root onto $G$ at the non-root vertex $v_i$, $T_1 \wedge S_1 \wedge S_2 \wedge \dots \wedge S_n$ is a cable-trench solution for $G\wedge H_{v_1} \wedge H_{v_2} \wedge \dots \wedge H_{v_n}$ as long as for all edges $e_2 \neq e_1 \in E(G)$, we have:

    \begin{align}\label{eq:multi_wedge_ineq}
    0 \geq \tau(e_2) - \tau(e_1) + &L(P_2) - L(P_1) + (L(P_2)-L(P_1))|\tilde{Q}| + C(\tilde{Q}^+) - C(\tilde{Q}^-)\\
    &+ \sum_{v\in H(\tilde{Q})} (L(R_{1}^{(v)}) - L(R_{2}^{(v)}))|H_v| \nonumber
    \end{align}
\end{theorem}

\begin{proof}
    The details of the proof are essentially parallel to the proof of Proposition \ref{prop:inequality}. Recall the decomposition of $G$ specified by $e_1$ and our arbitrary choice of $e_2\neq e_1$.

Like before, we will not prefer the tree deleting $e_2$ if $0 \geq cost(T_1 \wedge S_1\wedge \dots \wedge S_n) - cost(T_2 \wedge S_1\wedge \dots \wedge S_n)$. In this case, deleting $e_2$ would necessarily extend to a tree in the wedge graph which cannot be the cable-trench solution we desire.
    
When expanding the cost terms in the above inequality, the cost of extending each tree $T_1$ and $T_2$ into a potential solution within the full wedge $G\wedge H_{v_1} \wedge H_{v_2} \wedge \dots \wedge H_{v_n}$ is done analogously to Proposition \ref{prop:inequality}. The only difference we must examine more carefully is the contribution to the overall cost coming from needing to pre-cable each of the cable-trench trees $S_1, \dots, S_n$ from their respective roots backward to the root of $G$.
     
First, observe that for any graph $H_i$ for which $v_i$ is not on path $\tilde{Q}$, the pre-cabling cost necessary to work backward from the root of $H_i$ to the root of $G$ is identical regardless of which of $T_1$ or $T_2$ we use in $G$. Since, $v_i$ is not on $\tilde{Q}$, the path from $r_G$ to $v_i$ is the same in both trees. Therefore the additional cable-trench cost coming from wedging such an $S_i$ will cancel in the difference $cost(T_1\wedge S_1\wedge\dots \wedge S_n) - cost(T_2\wedge S_1\wedge \dots \wedge S_n)$. 

It thus follows that the only meaningful contributions to overall cable-trench cost arise when wedging trees $S_i$ at vertices $v_i \in \tilde{Q}$. The contribution of these graphs to the difference is precisely
    $$\sum_{v\in H(\tilde{Q})} (L(R_{1}^{(v)}) - L(R_{2}^{(v)}))|H_v|$$
since for each graph $H_i$ such that $v_i\in \tilde{Q}$, we must take the difference between the cost of pre-cabling $|H_i|$ cables along path $R_2^{(v_i)}$ compared to the cost of pre-cabling along path $R_1^{(v_i)}$. 

The same algebraic manipulations as in the proof of Proposition \ref{prop:inequality} then show that the inequality (\ref{eq:multi_wedge_ineq}) holds if and only if $T_1 \wedge S_1 \wedge S_2 \wedge \dots \wedge S_n$ is a minimum cable-trench tree as desired.
\end{proof}

In each case, finding the particular edge $e_2$ for which deleting $e_2$ would yield a cheaper total cable-trench cost involves (in the worst case) verifying the inequality over all possible values of $e_2 \in R_1$. 
This is captured by the loop condition in Algorithm \ref{alg:cap}.

Finally, with Theorem~\ref{thm:Wedging} we have some control on what happens when we wedge multiple graphs onto a single cycle $G$, at any number of non-root vertices. Specifically, in this construction we frequently continue to be able to use a local cable-trench solution $T_G$ in our search for a global cable-trench solution for the resulting wedge graph. Once again, in-depth analysis regarding when, and how, this particular local-to-global transfer of information fails will be explored in Section~\ref{section4}. In Section~\ref{section6} we will also use this property to inductively solve the cable-trench problem for a large collection of graphs made out of wedging arbitrarily many trees and cycles.


\section{The Strength Index of a Graph}{\label{section4}}

In Section~\ref{section3}, we repeatedly saw that, when wedging a graph $H$ onto $G$, the only additional cable-trench cost, coming from within $G$ itself, is directly related to the cabling cost from the root $r$ to the wedge vertex $v$. Thus, the total cable length of the $r$ to $v$ path contained in a candidate tree is key to understanding if that tree could indeed be a cable-trench solution in $G \wedge H$.

\begin{lemma}\label{lem:Myas}
Assume $G$ is a cycle with cable-trench solution $T_1$ and $H$ is an arbitrary graph with known cable-trench solution $T_H$.
If $L(P_1Q_1) > L(R_1)$, then there is no edge $e_2 \neq e_1$ such that $cost(T_2 \wedge T_H) < cost(T_1 \wedge T_H)$. 
\end{lemma}

\begin{proof}
By assumption, $T_1$ is a cable-trench solution internal to the graph $G$ so $cost(T_1) \leq cost(T_2)$ and thus, after expanding and collecting terms on one side of the inequality, the following must hold:
    \begin{align*}
    0 \leq \tau(e_1) - \tau(e_2) + &C(R_2) - C(R_1) + (C(P_2) - L(P_2)) - (C(P_1) - L(P_1))\\
    &+ L(R_2)|Q_2| - L(R_1)|Q_1| + C(Q_2) - C(Q_1)
    \end{align*}
Now, suppose that there exists an edge $e_2 \neq e_1$ with $cost(T_2 \wedge T_H) < cost(T_1 \wedge T_H)$. By Lemma~\ref{lem:BetterAlpha} $e_2 \in R_1$, hence $P_1Q_1 = R_2$. Then as a consequence of Proposition \ref{prop:inequality} we must have:
    \begin{align*}
    0 > \tau(e_1) - \tau(e_2) &+ C(R_2) - C(R_1) + (C(P_2) - L(P_2)) - (C(P_1) - L(P_1)) \\
    &+ (L(R_2)|Q_2|-L(R_1)|Q_1|) + C(Q_2) - C(Q_1) + (L(R_2) - L(R_1))|H|
    \end{align*}
    
    However, the only way to reconcile these two inequalities is for $L(R_2) - L(R_1) < 0$. That is $L(R_2) < L(R_1)$, which contradicts the assumption that $L(P_1Q_1) > L(R_1)$. Hence, no such $e_2$ can exist.
\end{proof}

What we now know is that, if $L(R_2) > L(R_1)$ then no matter how large the graph $H$ we may always use local cable-trench solutions $T_G$ in $G$ and $T_H$ in $H$ to explicitly construct a global cable-trench solution in $G \wedge H$.

Intuitively, Lemma~\ref{lem:Myas} relies on the fact that when switching to $T_2$ the only way we might reduce our overall cable-trench cost is by requiring a cheaper version of pre-cabling since we now do so through path $R_2$ instead of $R_1$. However, if $L(R_2) > L(R_1)$, then we cannot save on the pre-cabling costs and there is no other way to save elsewhere in our choice of spanning tree internal to $G$.
We formalize this concept in the idea of the \emph{strength} of a cycle $G$.

\begin{definition}
Assume $G$ is a cycle with cable-trench solution $T_1$. Let $v$ denote the wedge vertex and $e$ be an edge in $R_1$. The \textbf{edge strength} of this vertex-edge pair is denoted $\sigma(v,e)$. The value of $\sigma(v,e)$ is the size of the vertex set for the largest graph $H$ (with known cable-trench solution $T_H$) that can be wedged onto $G$, at $v$, such that $T_1 \wedge T_H$ is a cable-trench solution for $G \wedge H$. 
\end{definition}

As shown in Lemma \ref{lem:Myas}, if $L(R_2) > L(R_1)$, then $H$ can be any size, so we define $\sigma(v,e) = \infty$ for any such $(v,e)$ pair.

\begin{definition}
The \textbf{vertex strength} of a wedge vertex $v$ is $\sigma(v) = \min\{\sigma(v,e): e \in R_1\}$.
\end{definition}

The vertex strength is exactly the size of the largest graph that can be wedged onto $G$, at $v$, such that removing $e_1$ remains a correct choice to find a cable-trench solution for the composite wedge graph.

\begin{definition}
The \textbf{breaking edge} is the edge $e$ on $R_1$ with $\sigma(v,e) = \sigma(v)$.
\end{definition}

The breaking edge is called as such since it is the first edge to ``break'' when a large enough $H$ is wedged onto $G$. Specifically, it is the edge which should be removed, instead of $e_1$, in a cable-trench solution for the composite wedge graph.

In the definition of breaking edge, if multiple edges all attain the same minimum $\sigma(v,e)$ value, any of them may be chosen as the breaking edge. As always, we only use the breaking edge to attain minimum cable-trench cost without worrying about uniqueness of the specific tree.

\begin{lemma}\label{lem:breaking_exists}
Assume $G$ is a cycle rooted at $r$ with cable-trench solution $T_1$. If the wedge vertex $v$ is chosen so that $L(R_2) < L(R_1)$, then there exists a breaking edge $e_2$ on the path $R_1$.
\end{lemma}

\begin{proof}
    For any edge $e_2$ on $R_1$ and for arbitrary non-empty graph $H$ with internal cable-trench solution $T_H$, we consider when $cost(T_2 \wedge T_H) \leq cost(T_1 \wedge T_H)$. From Proposition~\ref{prop:inequality}, this would occur exactly when the following inequality is satisfied.
    \begin{align*}
0 < \tau(e_2) - \tau(e_1) + L(P_2) - L(P_1) + (L(P_2)-L(P_1))|\tilde{Q}| + C(\tilde{Q}^+) - C(\tilde{Q}^-) + (L(R_1) - L(R_2))|H|
    \end{align*}
    
    Notice that all of the terms on the right-hand side, except for the last, are fixed for a given choice of $e_1$ and $e_2$. 
    Since we assume $L(R_2) < L(R_1)$ and $|H| > 0$, the final term, $(L(R_1) - L(R_2))|H|$, is strictly positive. So if we choose $H$ such that $|H|$ is large enough to overcome any negative value the rest of the terms may attain, we are guaranteed to satisfy the desired inequality. The smallest such value of $|H|$ is then precisely $\sigma(v,e_2) + 1$. Specifically, any smaller value of $|H|$ would \emph{not} satisfy the inequality.
    
    Now optimizing over all $e_2 \in R_1$, any edge which attains the minimum $\sigma(v,e_2)$ is then, by definition, the breaking edge. 
\end{proof}

As a consequence of Lemma~\ref{lem:breaking_exists}, we know that when $L(R_2) < L(R_1)$ it follows that $\sigma(v)$ is finite. Moreover, the proof demonstrates that both $\sigma(v)$ and the associated breaking edge $e_2$ may be determined with only internal knowledge of $G$ itself. 

Finally, the breaking edge $e_2$ is shown to exist on $R_1$ and the following corollary of Lemma~\ref{lem:BetterAlpha} ensures that no ``secondary'' breaking edge could exist on $P_1 \cup Q_1$.

\begin{corollary}\label{cor:breaking_unique}
Assume $G$ is a cycle rooted at $r$ with cable-trench solution $T_1$. If the edge $e_2 \in R_1$ is the breaking edge for the wedge vertex $v$, there will never be an edge $e_3$ on $P_1\cup Q_1$ such that $cost(T_3 \wedge T_H) < cost(T_2 \wedge T_H)$
\end{corollary}

\begin{proof}
    Follows immediately from Lemma \ref{lem:BetterAlpha} since $e_2$ is the breaking edge, so we know $cost(T_2 \wedge T_H) < cost(T_1 \wedge T_H)$.
\end{proof}

With these results, we can now state the following extension to Proposition~\ref{prop:inequality}.

\begin{theorem}\label{thm:cycles}
Assume $G$ is a cycle with cable-trench solution $T_1$. Let $H$ be an arbitrary graph with known cable-trench solution $T_H$. We consider the wedge graph $G \wedge H$ where we wedge $H$ at its root onto $G$ at a non-root vertex $v$. Let $e_2$ denote the breaking edge associated to $v$.

Then, $T_1 \wedge T_H$ is a cable-trench solution for $G \wedge H$ if $|H| \leq \sigma(v)$ and $T_2 \wedge T_H$ is a cable-trench solution otherwise.
\end{theorem}

\begin{proof}
The definition of the strength $\sigma(v)$ means that $|H| \leq \sigma(v)$ implies the inequality of Proposition~\ref{prop:inequality} cannot be satisfied and so $T_1 \wedge T_H$ is a cable-trench solution. 

Lemma~\ref{lem:WedgeTreeRemains} and the definition of the breaking edge $e_2$ means that $|H| > \sigma(v)$ implies $T_2 \wedge T_H$ is a cable-trench solution.

Finally, Corollary~\ref{cor:breaking_unique} means there are no other possibilities we must consider.
\end{proof}

With Theorem~\ref{thm:cycles} we have shown that, when wedging a cycle $G$ at a non-root vertex, we may always use either a local cable-trench solution $T_G$ or an easily computed alternative local tree $T'_G$ (i.e. the tree obtained by removing the breaking edge) when constructing a global cable-trench solution in $G \wedge H$. 

That is, we now know that while a local solution to the cable-trench problem within $G$ is not always sufficient to find a global solution within $G \wedge H$, there is only a single other possibility for what the restriction to $G$ could be that we might need to consider. So, there are \emph{at most} two options for a cable-trench solution to $G \wedge H$ we must check.

In the broader algorithmic context, the strength indices and breaking edges of a graph $G$ are internal properties of $G$, and thus for a fixed graph $G$, can be precomputed. 
When wedging graph $H$ onto $G$, the spanning tree computation then reduces to a simple check of whether or not $|H|$ exceeds the strength of the wedge vertex, and if so, we immediately know an optimal spanning tree for the wedge graph via knowing the breaking edge. 
In essence, these internal graph properties offer even more time savings in Algorithm \ref{alg:cap} as the for loop is now unnecessary and the cost computation is drastically simplified. 

\section{The $\theta$-Graph}

In Sections~\ref{section3} and \ref{section4}, we established many cases where local information about any cable-trench solutions on individual components lead to global knowledge of a cable-trench solution on a resulting wedge graph. We present further evidence of the efficacy of this ``local-to-global'' notion by giving another class of graphs for which such a process leads to an efficient method for solving the generalized cable-trench problem. 

In this section, we will consider a variant of a cycle graph which is known as a $\theta$-graph. Such a graph contains two distinguished vertices -- a root vertex $r$ and another vertex $v$. These vertices are connected through 3 edge-disjoint paths. Observe that in a $\theta$-graph every spanning tree is uniquely determined by removing one edge each from two of those disjoint paths. In Figure~\ref{fig:ThetaTerms} we give an analogous decomposition of the $\theta$-graph as we had for cycles in Figure~\ref{fig:CycleTerms}. 

This particular labeling assumes two demarcated edges, $e_1$ and $e_2$, are used to decompose $G$ into subpaths pertinent to the implied spanning tree $T_{1,2}$ obtained by deleting $e_1$ and $e_2$. A third edge $e_3$ chosen along the path $R_{1,2}$ then gives two other possible iterations of this style of labeling: that with $e_1$ and $e_3$ demarcated and that with $e_2$ and $e_3$ demarcated.

\begin{figure}[hbtp]
    \figlabel{fig:ThetaTerms}
    \begin{center}
    \includegraphics[width=200pt]{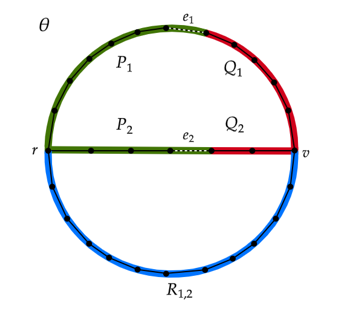}
	\caption{The $\theta$-graph $G$ with distinguished edges $e_1, e_2$ along with labeled components for the corresponding decomposition.}
    \end{center}

\end{figure}

\subsection{A Cable-Trench Solution on $\theta$-graphs}

Our knowledge of cycles and graphs constructed out of wedges with cycles is sufficient to directly solve the cable-trench problem on $\theta$-graphs. To show this, we sort all spanning trees of the $\theta$-graph $G$ into three overlapping groups.

\begin{enumerate}
\item $A$, those spanning trees which remove an edge $e_3$ on $R_{1,2}$.
\item $B$, those spanning trees which remove an edge $e_2$ on $R_{1,3}$.
\item $C$, those spanning trees which remove an edge $e_1$ on $R_{2,3}$.
\end{enumerate}

For the moment, consider spanning trees in group $A$. That is we are certain to exclude an edge $e_3$ on path $R_{1,2} = P_3 \cup Q_3$. Now, we consider the cycle $R_{1,3} \cup R_{2,3}$ and determine the cable-trench solution internal to this cycle as well as the vertex strength $\sigma(v)$ and the breaking edge $e$. Next we re-interpret the paths $P_3$ and $Q_3$ as trees wedged onto the cycle $R_{1,3} \cup R_{2,3}$ at $r$ and $v$ respectively.

By Proposition \ref{wedge_trees_at_roots}, wedging $P_3$ can cause no change internal to $R_{1,3} \cup R_{2,3}$ because it is directly wedged to the root vertex. To determine the effect of wedging $Q_3$, we must decide whether $|Q_3|$ exceeds $\sigma(v)$. Then, by Theorem~\ref{thm:cycles}, we know an optimal cable-trench cost tree for $(R_{1,3} \cup R_{2,3}) \wedge P_3 \wedge Q_3$ in any case. That is, we know an optimal cable-trench cost tree among all trees in $A$.

In a similar way, we may find optimal cable-trench cost trees in $B$ and $C$. Finally, one of these three optimal solutions must be an overall cable-trench solution for the original $\theta$-graph as desired.

\subsection{Analogs of the Strength Index and Breaking Edge in $\theta$-graphs}

We now define appropriate generalizations of the strength index and breaking edge so that we may carefully use $\theta$-graphs in our wedging construction. Specifically, we focus on what happens when we wedge an arbitrary graph $H$ onto the $\theta$-graph $G$ at the non-root vertex $v$.

In the remainder of this section we will assume, without loss of generality, that the internal cable-trench solution for the $\theta$-graph $G$ excludes fixed edges $e_1\in P_1\cup Q_1$ and $e_2\in P_2\cup Q_2$. Following previous conventions, we call the tree $T_{1,2}$.

First, we present an analog to Lemma~\ref{lem:BetterAlpha} which ensures we do not need to consider removing any other edges on $P_1\cup Q_1$ and $P_2 \cup Q_2$ in our search for a cable-trench solution for $G \wedge H$.

\begin{lemma}{\label{lem: better_alpha_theta}}
Given a $\theta$-graph $G$, suppose a cable-trench solution for $G$ is the tree $T_{1,2}$ that deletes edge $e_1$ on $P_1 \cup Q_1$ and edge $e_2$ on $P_2 \cup Q_2$ (see Figure \ref{fig:ThetaTerms}). Let $H$ be an arbitrary graph with cable-trench solution $T_H$ and consider the wedge $G \wedge H$ performed at the root $r_H$ of $H$ and the non-root vertex $v$ of $G$.

For any edges $e_1' \neq e_1$ on $P_1 \cup Q_1$ and $e_2' \neq e_2$ on $P_2 \cup Q_2$, consider the associated tree $T_{1,2}'$. We must have $cost(T_{1,2}' \wedge T_H) \geq cost(T_{1,2} \wedge T_H)$.
\end{lemma}

\begin{proof}
    The proof follows similar logic to the proof of Lemma \ref{lem:BetterAlpha}. 
    Given that a cable-trench solution for $G \wedge H$ excludes edges $e_1 \in P_1 \cup Q_1$ and $e_2 \in P_2 \cup Q_2$, the contribution to the overall cable-trench cost arising from wedging $T_H$ at vertex $v$ is $L(R_{1,2})|H| + cost(T_H)$
    since any path from the root $r$ to a vertex in $H$ must be pre-cabled through $R_{1,2}$.
    
    Indeed, any spanning tree excluding edges from $P_1 \cup Q_1$ and $P_2 \cup Q_2$ must route cables to $H$ through $R_{1,2}$ and so the contribution of pre-cabling $T_H$ is always the same.

     The conclusion follows immediately from this observation and recalling $T_{1,2}$ is a cable-trench solution internal to $G$.
\end{proof}

Note that Lemma~\ref{lem: better_alpha_theta} makes no general claims on the specific edges excluded in the cable-trench solution of $G \wedge H$. It merely says we cannot find a lower cable-trench cost tree for $G \wedge H$ which removes different edges on the exact same pair of paths ($R_{2,3}, R_{1,3},$ or $R_{1,2}$) from those edges removed when focusing solely on $G$.

To examine when, and how, we find cable-trench solutions for the wedge $G \wedge H$ that do not restrict to the presumptive cable-trench solution $T_{1,2}$ in $G$, we provide the following extensions of vertex strength for $\theta$-graphs.

\begin{definition}
Assume $G$ is a $\theta$-graph with cable-trench solution $T_{1,2}$.
Let $v$ denote the wedge vertex and $e_3$ be an edge in $R_{1,2}$. 

The \textbf{first edge strength} of this vertex-edge pair is denoted $\sigma_1(v, e_3)$. 
The value of $\sigma_1(v,e_3)$ is the size of the vertex set for the largest graph $H$ (with known cable-trench solution $T_H$) that can be wedged onto $G$, at $v$, such that $cost(T_{1,2}\wedge T_H) \leq cost(T_{1,3}\wedge T_H)$ and $ cost(T_{1,2}\wedge T_H) \leq cost(T_{2,3}\wedge T_H)$ both hold.

The \textbf{second edge strength} of this vertex-edge pair is denoted $\sigma_2(v, e_3)$. 
The value of $\sigma_2(v,e_3)$ is the size of the vertex set for the largest graph $H$ (with known cable-trench solution $T_H$) that can be wedged onto $G$, at $v$, such that at least one of $cost(T_{1,2}\wedge T_H) \leq cost(T_{1,3}\wedge T_H)$ and $cost(T_{1,2}\wedge T_H) \leq cost(T_{2,3}\wedge T_H)$ holds.
\end{definition}

As before, we use these edge strengths to focus on the wedge vertex by minimizing over all $e \in R_{1,2}$.

\begin{definition}
The \emph{first vertex strength} of a wedge vertex $v$ is $\sigma_1(v) = \min_{e\in R_{1,2}}\sigma_1(v, e)$.

The \emph{second vertex strength} of a wedge vertex $v$ is $\sigma_2(v) = \min_{e\in R_{1,2}}\sigma_2(v, e)$
\end{definition}

We note that $\sigma_1(v)$ is the size of the vertex set for the largest graph $H$ (with known cable-trench solution $T_H$) that can be wedged on the $\theta$-graph $G$, at $v$, such that $T_{1,2} \wedge T_H$ is a cable-trench solution for $G \wedge H$.

It then follows that, when $|H| > \sigma_1(v)$ the cable-trench solution to $G\wedge H$ must restrict to some tree other than $T_{1,2}$ in $G$. By Lemma~\ref{lem: better_alpha_theta}, this new tree must remove an edge $e_3$ on $R_{1,2}$. However, it is unclear whether it would also remove an edge on $R_{1,3}$ or if it would also remove an edge on $R_{2,3}$ instead. In other words, a cable-trench solution to $G \wedge H$ might restrict to a tree of type $T_{1,3}$ or it might restrict to a tree of type $T_{2,3}$ in $G$. 

Indeed, there may be trees of both type $T_{1,3}$ and of type $T_{2,3}$ which are not cable-trench solutions in $G$, but are restrictions of cheaper cable-trench cost options in the wedge graph $G \wedge H$. The second vertex strength $\sigma_2(v)$ exactly measures the size of the vertex set for the largest graph $H$ before we might need to swap between these two possibilities. For example, we might have that for $|H| \leq \sigma_2(v)$ we find $T_{1,3} \wedge T_H$ is a cable-trench solution for $G \wedge H$ while for $|H| > \sigma_1(v)$ we find $T_{2,3} \wedge T_H$ is a cable-trench solution for $G \wedge H$.

To formalize these observations, we provide a generalization of Lemma \ref{lem:Myas}.

\begin{lemma}\label{lem:Myas_theta}
Assume $G$ is a $\theta$-graph with cable-trench solution $T_{1,2}$ and wedge vertex $v$. Then, for arbitrary choice of $e_3 \in R_{1,2}$ and
    \begin{enumerate}
        \item if both $L(R_{2,3}) > L(R_{1,2})$ and $L(R_{1,3}) > L(R_{1,2})$ hold, then $\sigma_1(v) = \sigma_2(v) = \infty$;
        \item if exactly one of $L(R_{2,3}) > L(R_{1,2})$ or $L(R_{1,3}) > L(R_{1,2})$ hold, then $\sigma_1(v) < \sigma_2(v) = \infty$;
        \item if neither $L(R_{2,3}) > L(R_{1,2})$ nor $L(R_{1,3}) > L(R_{1,2})$ hold, then $\sigma_1(v) \leq \sigma_2(v) < \infty$.
    \end{enumerate}
\end{lemma}

\begin{proof}
Assume $H$ is an arbitrary graph with known cable-trench solution $T_H$.

First we prove (1) by showing that $L(R_{1,3}) > L(R_{1,2})$ implies that there is no edge $e_3 \in R_{1,2}$ with $cost(T_{1,3}\wedge T_H) < cost(T_{1, 2}\wedge T_H)$. A similar argument will show that if $L(R_{2,3}) > L(R_{1,2})$ then there is no $e_3\in R_{1,2}$ such that $cost(T_{2,3}\wedge T_H) < cost(T_{1,2}\wedge T_H)$. 

The total cost of $T_{1,2} \wedge T_H$ is
\begin{align*}
    cost(T_{1,2} \wedge T_H) = \tau(E(G))-\tau(e_1)-\tau(e_2)&+ C(P_1^+) - L(P_1) +C(P_2^-) - L(P_2) \\ + C(R_{1,2})+L(R_{1,2})(|H|+|Q_1|+|Q_2|)
    &+ C(Q_1^-) + C(Q_2^-) + cost(T_H)
\end{align*}

where we further decompose $C(R_{1,2})$ as follows
$$C(R_{1,2}) = C(P_3^-) + L(P_3)|Q_3| + C(Q_3^-)$$.

The total cost of $T_{1,3}\wedge T_H$ is similarly computed as
\begin{align*}
    cost(T_{1,3} \wedge T_H) = \tau(E(G))-\tau(e_1)-\tau(e_3) &+ C(P_1^+) - L(P_1)+C(P_3-)-L(P_3) \\+ C(R_{1,3})+L(R_{1,3})(|H|+|Q_1|+|Q_3|)
    &+ C(Q_1^-) + C(Q_3^+) + cost(T_H)
\end{align*}

and again we decompose $C(R_{1,3})$ as
$$C(R_{1,3}) = C(P_2^+) +  L(P_2)|Q_2| + C(Q_2^+)$$.

If $cost(T_{1,3} \wedge T_H)$ is eventually cheaper, then the difference of the two expressions $cost(T_{1,2}\wedge T_H) - cost(T_{1,3}\wedge T_H)$ should be positive. After cancelling like terms in both expressions, this implies that for arbitrary $|H|$
\begin{align}\label{theta_ineq}
    &(\tau(e_3)-\tau(e_2))\\
    &+ L(P_3) + L(P_3)|Q_3|+L(P_3)|Q_3| +L(R_{1,2})(|H|+|Q_1| + |Q_2|)\\
    &-L(P_2)- L(P_2)|Q_2| -L(P_2)|Q_2| - L(R_{1,3})(|H|+|Q_1| + |Q_3|) \\
    &+ C(Q_2^-) - C(Q_2^+) + C(Q_3^-)-C(Q_3^+) > 0
\end{align}

As a side note, we see that parallel to Proposition~\ref{prop:inequality}, the terms in (2) account for the difference in trench lengths, (3) and (4) account for the difference in contributions from paths $R_{1,3}$ and $R_{1,2}$, and (5) accounts for the differently oriented regions $Q_2$ and $Q_3$. 

Remember, we assume that $T_{1,2}$ is an internal cable-trench solution for $G$. Therefore, when $|H|=0$, we must have that $cost(T_{1,2} \wedge T_H) - cost(T_{1,3} \wedge T_H) \leq 0$. 
From this, we can say that when $|H| > 0$, if the inequality~\ref{theta_ineq} is to be satisfied, we \textit{must} have $(L(R_{1,2}) - L(R_{1,3}))|H|>0$.

Hence, in order for there to exist a suitable $e_3$, we need that $L(R_{1,2}) > L(R_{1,3})$. This contradicts our assumption, and so no such $e_3$ exists.

If we apply the same arguments to the tree $T_{2,3}$, we get that, in order for there to exist an $e_3$ with $cost(T_{2,3} \wedge T_H) < cost(T_{1,2} \wedge T_H)$, we must have $L(R_{1,2}) > L(R_{2,3})$. Again, contradicting our assumptions.

So, if both $L(R_{2,3}) > L(R_{1,2})$ and $L(R_{1,3}) > L(R_{1,2})$ for all $|H|$, then there is no cable-trench solution for $G \wedge H$ that restricts to either type $T_{1,3}$ or type $T_{2,3}$. Which gives $\sigma_1(v) = \sigma_2(v) = \infty$ by definition.

In order to prove (2), we apply the rationale in the proof of Lemma~\ref{lem:breaking_exists} to exactly one of the pairs $R_{2,3}, R_{1,2}$ or $R_{1,3}, R_{1,2}$.

For example, say that $L(R_{2,3}) > L(R_{1,2})$. Then, for each choice of $e_3 \in R_{1,2}$ we can directly solve for the smallest value of $|H|$ such that $cost(T_{2,3} \wedge T_H) < cost(T_{1,2} \wedge T_H)$. Then minimize over all $e_3\in R_{1,2}$ to find finite $\sigma_1(v)$.

Since we also assume $L(R_{1,3})\leq L(R_{1,2})$, by the same logic as the proof of part (1), this must give $\sigma_2(v) = \infty$.

Finally, we prove (3) by applying the rationale in the proof of Lemma~\ref{lem:breaking_exists} to both pairs $R_{2,3}, R_{1,2}$ and $R_{1,3}, R_{1,2}$. This, coupled with the fact that $\sigma_1(v) \leq \sigma_2(v)$ by definition, gives the desired result.
\end{proof}

Throughout the preceding proof, our attention was often focused on the the existence of a specific edge $e_3 \in R_{1,2}$. This edge is the $\theta$-graph version of a \emph{breaking edge}. However, there may be more than one such edge in a $\theta$-graph since there are two distinct strengths corresponding to potentially different emergent trees depending on the size $|H|$.

Just the same, the fact that $\theta$-graphs admit tractable solutions is essentially a direct consequence of the results of Sections~\ref{section3} and \ref{section4}. The existence of concrete strength analogs means $\theta$-graphs may be analyzed using similar efficiency to Algorithm~\ref{alg:cap}.

We conjecture that iterating upon the previous constructions should allow us to compute the cable-trench solution for a graph in which two vertices are connected by 4 (or more) edge-disjoint paths. This would mean such graphs are tractable as well by a similar argument to the one presented at the beginning of the section. Once again, it seems that the idea of strengths should enable us to further extend the class of graphs for which we can quickly compute the cable-trench solution.

\section{Intractability of CTP in General Graphs }\label{section6}
The generalization of the strength precomputation to $\theta$-graphs demonstrates a promising `inductive' use of strength in order to increase the class of graphs that can be extended. Just the same, the increased complexity does present itself in the increasingly expensive precomputation of (multiple) strength statistics about the graph. 
Further, the class of graphs attainable using the above methods is evidently quite delicate; for example, the methods described above do not easily generalize when the wedge vertex $v$ lies on only one of the three paths, $R_{2,3}$, $R_{1,3}$, or $R_{2,3}$, in the $\theta$-graph.

We now return to our discussion of graphs for which we can compute the cable-trench solution quickly. We refer to such graphs as \textbf{tractable}.

\subsection{Tractability in cactus graphs}
As shown in Section~\ref{section3}, if we are given a collection of cycles and trees, we can iteratively construct a graph $G$ by wedging together these pieces and, in parallel, keeping track of cable-trench solution. Furthermore, this can be done in polynomial time.
Such a construction ensures the resulting graph, in which all cycles are edge-disjoint (i.e. each edge is in at most one cycle), is tractable. Such graphs are called \emph{cactus graphs}. 

Decomposing a cactus graph into it's component cycles and trees may be done in linear time via the \emph{biconnected components} algorithm outlined in \cite{Tarjan}. As the name suggests, the graph is decomposed into biconnected components which (for a cactus graph) correspond to cycles and single edges. Thus we can efficiently find a \emph{maximal} decomposition, one in which every pair of components intersects in at most one vertex.
 
Hence, given a cactus graph, following a linear time preprocessing step of the graph into it's component pieces, we may iteratively construct the cable trench solution as Theorem \ref{eq:multi_wedge_ineq} suggests. 
\begin{corollary}
Cactus graphs are tractable.
\end{corollary}

\begin{figure}[hbtp]
    \begin{center}
    \includegraphics[width=250pt]{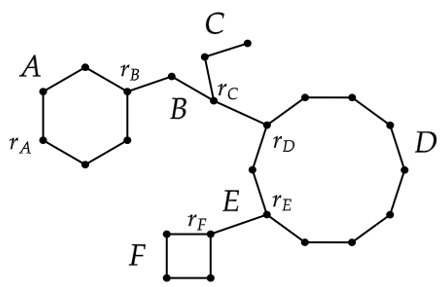}
    \caption{Illustrative example of the biconnected component decomposition (or \textit{`block-cut tree' decomposition} of a cactus graph into cycles $A, D, F$ and trees $B, C, E$.} 
    \end{center}
\figlabel{fig:example}
\end{figure} 
    
\begin{example} \label{ex_cactus}
Consider the graph in Figure \ref{fig:example}. If we begin with tree $E$ rooted at $r_E$, we can wedge cycle $F$ at $r_F$, and thus determine the cable-trench solution for $E \wedge F$ rooted at $r_E$. At this point, we could wedge that result to $D$ at $r_E$, and continue in this way until we wedge the graph $B \wedge C \wedge D \wedge E \wedge F$ onto $A$ at vertex $r_B$. 
At each step\textcolor{red}{,} we know the cable-trench tree, which allows us to wedge it onto another tree or cycle iteratively and efficiently. 
   
\end{example}

\begin{observation}
Let $T$ be a spanning tree for a cactus graph $G$.
Given any component $G'$ of the decomposition of $G$ into paths and cycles, the restriction of $T$ to $G'$ is necessarily a spanning tree of $G'$.
\end{observation} 
Essentially, our knowledge of local spanning trees for individual components of $G$ determines the total number of relevant spanning trees in the entire graph.

\subsection{Generalizing the notion of strength}

The key observation in our proof of the tractability of cactus graphs is that precomputation of strength indices for each individual cycle is the single crucial piece of information necessary for the iterative construction of a cable-trench solution for the full cactus graph. Moreover, those strength indices can be computed completely intrinsically to each cycle and in polynomial time. Our discussion of the $\theta$-graph followed a similar road map.

In the following, we outline a method of defining strength indices in a general graph. This would, in theory, provide the necessary information for a more general iterative approach to constructing cable-trench solutions. However, we will indicate why determination of the exact values for such strength indices must be computationally difficult. Hence, we will directly tie the notion of intractability to the notion of vertex strength.

Consider a graph $G$ with root vertex $r$ and wedge vertex $v$. For every root path $P$ from $r$ to $v$, there is a minimum cable-trench cost tree, denoted $T(P)$, containing said path. We order this collection of minimal trees from smallest to largest in terms of total cable-trench cost and denote the resulting sequence of spanning trees $T(P_1), \dots , T(P_n)$. Note, by definition $T(P_1)$ is the cable-trench solution for $G$ though, importantly, the cable length $L(P_1)$ may not be the minimum among all root paths $P_i$.

\begin{conjecture} (cf. Lemma~\ref{lem:BetterAlpha} and Lemma~\ref{lem: better_alpha_theta}) Given the graphs $G$ and $H$, where $H$ has known cable-trench solution $T_H$. Let $T$ denote a cable-trench solution for the wedge graph $G \wedge H$. If $T$ contains the root path $P_i$ from $r$ to $v$ in $G$, then $cost(T|_G \wedge T_H) \geq cost(T(P_i) \wedge T_H)$.
\end{conjecture}

Basically, the tree $T(P_i)$ internal to $G$ is as good as any candidate for the possible restriction of a cable-trench solution in $G \wedge H$. That is, the local information about a solution in $G$ provides information about a solution in $G \wedge H$.

Recall that $L(P)$ is the cable length of the root path $P$. Beginning with $T(P_1)$, we create a subsequence of the spanning trees $T(P_1), \dots, T(P_n)$ where the root path associated to each successive tree has a shorter cable length than the previous. That is, we shorten our sequence to include only a list of spanning trees, still increasing with respect to total cable-trench cost, but now with $L(P_i) > L(P_{i+1})$ for all $i$ as well. We denote this new subsequence by $M_1 (= T(P_1) ), \dots, M_m$.

We then define the $i^{th}$ vertex strength $\sigma_i(v)$ to be the largest size $|H|$ for which a cable-trench solution to $G \wedge H$ restricts to the spanning tree $M_i$ in $G$. Since the sequence $M_1, \dots, M_m$ has decreasing values of $L(P_i)$, we make the following conjecture.

\begin{conjecture} (cf. Lemma~\ref{lem:Myas} and Lemma~\ref{lem:Myas_theta})
For the above vertex strengths $\sigma_i(v)$, we have $\sigma_1(v) \leq \sigma_2(v) \leq \dots \leq \sigma_m(v) \leq \infty$.
\end{conjecture}

Essentially, there is some finite size $|H|$ for which the additional cost of changing the internal spanning tree within $G$ from $M_i$ to $M_{i+1}$ is outweighed by the savings due to repetitive use of the shorter cable path $P_{i+1}$ (compared to $P_i$) when pre-cabling $T_H$.

Our argument for the efficiency of wedging onto cycles, $\theta$-graphs, and even cactus graphs would then boil down to noting that in all three cases the sequence of spanning trees $T(P_1), \dots, T(P_n)$, and the root-path cable-length decreasing subsequence $M_1, \dots, M_m$, are both computable in polynomial time. Indeed, there are only two possible root paths in the cycle (and three in the $\theta$-graph) each of which has a linear number of potential trees over which we must minimize. So, we can rapidly compute the necessary values of strength and thus streamline the subsequent wedge constructions.

This generalized definition of strengths would extend the ideas we previously showed gave us the highly useful local-to-global induction method. Furthermore, strength is now directly tied to the transition from one internal cable-trench minimizer to another. This formulation appears to mirror arguments employed by Vasko et al. \cite{Vasko2} with respect to the relative size of the cable and trench weights. This formulation also directly ties our tractability definition to the difficulty of finding a minimum cost spanning tree with a specific root path from $r$ to $v$. That is, the specific minimum-cost tree $T(P_i)$ associated to $P_i$ is difficult to find, even among only those spanning trees already containing $P_i$ (again, as shown in \cite{Vasko1}). So, computing \emph{all} the strengths for a general graph becomes $NP$-hard in its own right.

\subsection{Identifying multiple vertices in graph constructions}

We have seen that for cycles, trees, restricted cases of $\theta$-graphs, and graphs constructed by wedging the previous types under specific conditions, the cable-trench solution can be found quickly.
In general, there is a brute-force computation over all spanning trees via Kirchhoff's result on spanning trees that enables a fast search over a reasonably small set of spanning trees. 
A natural next step is to consider the point at which finding the cable-trench solution becomes computationally intractable. 

Of course, a necessary condition for a family of graphs to be \textit{intractable} is for the family to admit an exponential number of spanning trees. 
The example we will use is the $2\times m$ grid graph ($G[2,m]$). By a simple inductive argument, it is clear that the number of spanning trees of $G[2,m]$ grid graph is $\Omega(2^n)$.

\begin{figure}[hbtp]
    \begin{center}
    \includegraphics[width=350pt]{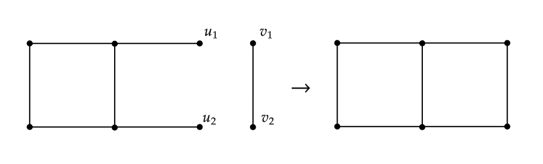}
    \caption{Example of $2\times m$ grid graph} 
    \end{center}
\figlabel{fig:grid}
\end{figure} 

Interestingly, we note that there exist spanning trees $T$ of the $2\times (m+1)$  grid such that the restriction $T'$ in the $2\times m$ grid is not a tree. 
This graph does not satisfy the property established in Example \ref{ex_cactus}, and as a consequence, an algorithm dependent on decomposing $G[2, m]$ should not be expected to consistently find a cable-trench solution. 
 
Because iteratively wedging graphs together at chosen roots preserves the tractability of the graph, and any arbitrary graph admits a path decomposition, we must have a connection (or a class of connections) between graph components for which the graph is no longer tractable. 
We conjecture that families of graphs with exponentially many spanning trees, and those whose construction necessitates identifying paths and cycles at multiple points, do not typically admit tractable solution schemes.


\bibliography{references}
\bibliographystyle{plain}

\end{document}